\documentclass[11pt]{article}
\usepackage{graphicx,epsfig}

\usepackage{amsmath,amsthm,amstext,amssymb}

\newtheorem{theorem}{Theorem}

\newtheorem{corollary}[theorem]{Corollary}

\theoremstyle{remark}
\newtheorem{example}{Example}
\newtheorem*{remark}{Remark}

\newcommand{\pair}[1]{\langle #1\rangle}
\newcommand{\bo}{\{0,1\}}
\def\phi{\varphi}
\def\le{\leqslant}
\def\ge{\geqslant}
\newcommand{\eps}{\varepsilon}

\newcommand{\CT}{\textit{CT}}
\newcommand{\sing}{\textrm{\rm Singl}}

\newcommand{\eqdef}{\rightleftharpoons}

\title{Short lists with short programs for functions}
\author{Nikolay Vereshchagin\thanks{The work was done while visiting
IMS (University of Singapore), the program ``Algorithmic Randomness'', 2--30 June 2014.
The work
was in part supported by the RFBR grant 12-01-00864.}\\
\\ Moscow State University, Yandex and Higher School of Economics}

\date{}

\begin{document}
\maketitle

\begin{abstract}
Let $\{\phi_p\}$ 
be an optimal G\"odel numbering of the family of computable functions (in Schnorr's sense), where $p$ ranges over binary strings. 
Assume that a list of strings $L(p)$ is computable from $p$ and for all $p$  
contains a $\phi$-program for $\phi_p$ whose length is at most $\eps$ bits larger
that the length of the shortest  $\phi$-program for $\phi_p$. We show that for infinitely many $p$ 
the list  $L(p)$ must have $2^{|p|-\eps-O(1)}$ strings. Here $\eps$ is an arbitrary function of $p$.
\end{abstract}

\section{Results}

A \emph{numbering of a family of computable functions of $m$ variables} 
is a computable partial function $\phi:(\bo^*)^{m+1} \to \bo^*$. 
We call $p$ a \emph{$\phi$-index} or  
a \emph{$\phi$-program} for the function $\pair{x_1.\dots,x_m}\mapsto \phi(p,x_1,\dots,x_m)$, which is
denoted as $\phi_p$.
A numbering  $\phi$ is \emph{universal} if 
for all computable partial functions
$f$ from  $(\bo^*)^m$ to $\bo^*$ there is $p$ with $\phi_p=f$. 

By $C_\phi(f)$ we denote the minimal length of a $\phi$-program for $f$ (\emph{Kolmogorov complexity
of $f$ with respect to $\phi$}).
A numbering  $\phi$ has \emph{Kolmogorov property},
if for every other numbering $\psi$ there is a constant $c$ such that 
 $C_\phi(f)\le C_\psi(f)+c$ for all functions $f$.

A numbering $\phi$ is called a \emph{G\"odel numbering} 
if for every other numbering $\psi$ there is a total computable function 
$t$ (called a \emph{translator} from $\psi$ to $\phi$) such that  
$\psi_p=\phi_{t(p)}$ for all $p$. 
A G\"odel numbering  $\phi$ is called an \emph{optimal G\"odel numbering} 
if for all numberings $\psi$ there is a translator $t$  from $\psi$ to $\phi$ that has additional property  
$|t(p)|\le|p|+O(1)$ (the translator is \emph{linearly bounded}).\footnote{
The term ``optimal G\"odel numbering''  was introduced by Schnorr~\cite{schnorr}. Teutsch and Zimand~\cite{tz}
call optimal G\"odel numberings \emph{Kolmogorov numberings}. However Kolmogorov has neither introduced nor 
studied them.} Here and further $|p|$ denotes the length of $p$.
Every optimal G\"odel numbering has Kolmogorov property but not the other way around.

\begin{example}\label{ex1}
Here is an example of an  optimal G\"odel numbering $\phi$ 
of the family of computable functions of $m$ variables.
Let $\Phi$ denote a universal 
numbering of the family of computable functions of $m+1$ variables.
Let $p\mapsto \hat p$ denote a computable prefix encoding,
for instance, $\hat p=0^{|p|}1p$.
Then $\phi(\hat pq,x_1,\dots,x_m)=\Phi(p,q,x_1,\dots,x_m)$ is an optimal G\"odel numbering  
of the family of computable functions of $m$ variables.
Indeed, the mapping $t(q)=\hat pq$ is a linearly bounded translator
from the numbering $\Phi_p$ to $\phi$.
\end{example} 

The above  
definitions make sense also for $m=0$. In this case $\phi_p$ is understood as $\phi(p)$ if defined and as a special symbol
$\bot$ otherwise. Optimal G\"odel numberings for $m=0$ were called \emph{standard machines} in~\cite{bmvz} and we will use
the same terminology. 
Kolmogorov complexity $C_U(x)$ of a string $x$  with respect to a standard machine $U$ is the usual Kolmogorov complexity 
(the minimal length of a $U$-program for $x$).

The paper~\cite{bmvz}
shows that for every standard machine $U$, 
given a string $x$ we can find a short list of strings with a short program for $x$:
the size (=cardinality) of the list is $O(|x|^2)$ and it contains 
a $U$-program for $x$ of length at most 
$C_U(x)+O(1)$. 

Is there a total algorithm that computes a short list with a short program 
for $x$ from any $U$-program for $x$?
This question was  asked recently by Alexander Shen~\cite{s}. 
Notice that there is no total algorithm  that maps any program for $x$ 
to $x$ (otherwise the positive answer to the question would immediately follow from the 
cited result of~\cite{bmvz}). We show
that for every standard machine $U$ and for every function $\eps$ of $p$
there are infinitely pairs ($x$, its $U$-program $p$) such that the size 
of $L(p)$  is exponential in both $|x|-\eps$ and $|p|-\eps$ provided $L(p)$
has a program for $x$ of length at most $C_U(x)+\eps$.

Let $C_{U,L}(x)$ denote the minimal length of a $U$-program $p\in L$ for $x$.

\begin{theorem}\label{th1}
Let  $U$ be a standard machine and $L$ a total computable function mapping (binary) strings 
to finite sets of strings. 
Then for some $c$ for all $k$ the following holds. 
There is a string  $x$ and its $U$-program $p$ of length 
between $k$ and $k+c$ such that
$\# L(p)\ge 2^{|x|}-2$ and   $C_{U,L(p)}(x)\ge k$.
\end{theorem}

\begin{corollary}[A negative answer to Shen's question] 
Let $U,L$ be as in the theorem.
Let $\eps(p)$ denote $C_{U,L(p)}(U(p))-|U(p)|$.
Then for infinitely many $p$ the size of $L(p)$ is at least $2^{|p|-\eps(p)-O(1)}-2$.
Moreover, for those $p$'s the size of $L(p)$ is at least $2^{|U(p)|}-2$ 
and $|U(p)|>|p|-\eps(p)-O(1)$.
\end{corollary}

Notice, that Kolmogorov complexity is less than the length (up to an additive constant)
and hence the corollary holds for $\eps(p)=C_{U,L(p)}(U(p))-C_U(U(p))$ as well.

\begin{proof}[Proof of the corollary]
Let $p,x$ be the pairs existing by Theorem~\ref{th1}. 
The last inequality of the theorem implies
$|x|+\eps(p)=C_{U,L(p)}(x)\ge k=|p|+O(1)$ and hence $|x|\ge |p|-\eps(p)-O(1)$. 
\end{proof}

Let us stress that $L$ is assumed to be a total function. If we allowed 
$L$ to be defined only on those strings $p$ for which  $U(p)$ halts,
then there would be a computable list  $L(p)$ of quadratic size (in the length of
$x=U(p)$)
 with a program for $x$ of length at most $C_U(x)+O(1)$, which follows from 
the result of~\cite{bmvz}.

\begin{example}
This example provides a family of computable lists for which 
the lower bounds for the size of $L$ and for $C_{U,L}(x)$ 
established in Theorem~\ref{th1} are tight.

The lower bound $k=|p|+O(1)$ for $C_{U,L}(x)$  is tight (up to an additve 
constant) 
for any list $L(p)$ containing $p$, for instance, 
for $L(p)=\{p\}$. For this list 
the lower bound for the size is tight too,
however, this is not very impressive, as the list is too small.

There is much larger computable list 
$L(p)=\{p\}\cup\bo^{<|p|}$ for which  
both lower bounds are tight. Indeed,
the length of the string $x$ in the theorem is $|p|+O(1)$, as
$C_{U,L(p)}(x)=C_{U}(x)\le |x|+O(1)$ and on the other hand
$C_{U,L(p)}(x)\ge k=|p|+O(1)$.

Moreover, there are similar lists of any log-caridanilty 
between 0 and $|p|$. 
Fix any computable function $p\mapsto i\le|p|$ 
and consider the computable list
$L_{i}(p)=\{p\}\cup\bo^{<i}$. 
For this list we have $\# L(p)=2^i$ and
$C_{U,L(p)}(x)=C_U(x)$ if $i>C_U(x)$ and 
$C_{U,L(p)}(x)=|p|$ otherwise.

The parameters ($\log\#L_i(p)$, $C_{U,L_i(p)}(x)$) for these lists are shown 
in the following picture (where we drop the subscript $U$):
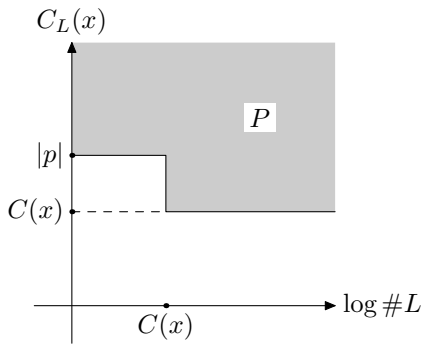
\begin{center}
\includegraphics[scale=1.0]{prof.eps}
        \end{center}
\label{p1}
More specifically,
they lie on the horizontal straight line segments on the border  of 
the gray area $P$.

Let us show that the lower bound of the size in the theorem is tight for all
computable lists of the form $L_i(p)$. That is, we will show 
that the length of the string $x$ existing by the theorem is $i+O(1)$.
As $2^{|x|}-2\le \#L_i(p)=2^i$, we have 
$|x|-1\le i$ and hence 
$|x|-1\le i\le |p|=C_{U,L_i(p)}(x)+O(1)$.
If $i\le  C_{U}(x)$ then we have 
$|x|-1\le i\le C_{U}(x)$ and hence  these inequalities are equalities up to an additive constant. Otherwise $i> C_{U}(x)$ and hence $C_{U,L_i(p)}(x)=C_U(x)$.
In this case  $|x|-1\le i\le C_{U,L_i(p)}(x)+O(1)=C_U(x)+O(1)$ and again 
 these inequalities are equalities up to an additive constant. 
\end{example}

Theorem~\ref{th1} easily translates to 
optimal G\"odel numberings of functions of arbitrary number of variables.
For general case the statement is the following. Let  $\sing_x$ denote the function 
defined only on the tuple $\pair{x,\dots,x}$ with value $x$. 
Let $C_{\phi,L}(f)$ denote the minimal length of a $\phi$-program $p\in L$ for $f$.

\begin{theorem}\label{th1b}
Let  $\phi$ be an optimal G\"odel function of $m>0$ variables and $L$ a total computable function mapping strings 
to finite sets of strings. 
Then for some
$c$ for all $k$ the following holds. 
There is a string $x$ and a $\phi$-program $p$ of length between $k$ and  $k+c$ 
for the function $\sing_x$ such that  
$\# L(p)\ge 2^{|x|}-2$ and  
$C_{\phi,L(p)}(\sing_x)\ge k$.
\end{theorem}   

\begin{remark} 
Theorem~\ref{th1b} holds for numberings of enumerable sets with the singleton set $\{x\}$ is place of the function $\sing_x$.
The proof is entirely similar.
\end{remark}

For the string $p$ from Theorem~\ref{th1b} we have 
\begin{equation}\label{eq2}
\log\#L(p)+C_{U,L(p)}(\phi_p)\ge C_U(\phi_p)+|p|-O(1). 
\end{equation}
Indeed, $\log\#L(p)\ge |x|-O(1)\ge C_U(\phi_p)-O(1)$ and
$C_{U,L(p)}(\phi_p)\ge k\ge |p|-O(1)$. Summing these inequalities we get~\eqref{eq2}.

Theorem~\ref{th1b} answers a question
asked recently by  Teutsch and Zimand.
For a numbering $\phi$ of computable functions of one variable, Teutsch and Zimand~\cite{tz}
considered the set of minimal programs for $\phi$,
where $p$ is called \emph{minimal}, if
for all $q<p$ we have $\phi_q\ne\phi_p$. Here $<$ denotes the lexicographical ordering
on binary strings (more precisely, $p<q$ iff $|p|<|q|$ or $|p|=|q|$ and $p$ is lexicographically less than $q$). 
The minimal $\phi$-program for a function $\phi_q$ is denoted 
by $\min_{\phi} (q)$. Teutsch and Zimand showed the following.
\begin{itemize}
\item If $\phi$ is a G\"odel numbering and a computable function $L$ on input
$p$ returns a list $L(p)$ containing $\min_\phi(p)$, then the size of
that list cannot be constant.
\item
For every numbering $\phi$ with Kolmogorov property, if a computable
function $L$ on input $p$ returns a list containing $\min_\phi(p)$, 
then the size of the list must be $\Omega(|p|^2 )$.
\item There exists an optimal G\"odel numbering $\phi$ such that if a computable
function on input $p$ returns a list containing $\min_\phi(p)$, then the size of that list must be $\Omega(2^{|p|})$.
\end{itemize}
In summary, their results show that a computable list that contains the minimal $\phi$-program cannot be too small. 

Along the lines of the second result Teutsch and Zimand 
asked the following question: is there an optimal G\"odel numbering $\phi$ with a computable
list $L(p)$ that contains $\min_\phi(p)$ and has size $O(|p|^2)$?

Theorem~\ref{th1b}  implies the negative answer to this question. Indeed, if 
$ \min_\phi(p)\in L(p)$ for all $p$ then $C_{\phi, L(p)}(\phi_p)=C_{\phi}(\phi_p)$ for all $p$.
By~\eqref{eq2} for the pair  $p,x$ existing by the theorem 
the size of $L(p)$  must be at least $2^{|p|-O(1)}$.
In other words, the third result of  Teutsch and Zimand 
holds for \emph{all} optimal G\"odel numberings $\phi$.

So far we were constructing for a given computable function $L$
inputs $p$ such that the list $L(p)$ has large parameters
$\#L(p)$ and $C_{U,L(p)}(U(p))$. Let us consider the ``short list with short programs'' problem
from the other end. Are there $p$'s such that 
\emph{every} short  list $L$ computable from $p$ by a total algorithm has high parameters 
$\#L(p)$ and $C_{U,L}(U(p))$? In this form the question is trivial:
we can hard-wire the shortest $U$-program $q$ for  $U(p)$ 
into a total algorithm which will return the list $\{q\}$, which has optimal  parameters.
The question becomes reasonable if 
we restrict the complexity, say by $O(\log|p|)$, of the total algorithm producing the list from 
$p$. 

To make this question
precise consider the \emph{total complexity}  $\CT_\Phi(a|b)$ 
defined as the minimal length of a $\Phi$-program of a total function that maps $b$ to $a$.
Here $\Phi$ is an optimal G\"odel numbering of computable functions of one variable. 

Fix a natural  $\delta$ (the upper bound for total complexity).
Then for each $p$ consider the set
$$
S^{\delta}_p=\{(i,j)\mid \exists L,\ \CT(L|p)\le\delta,\ \#L\le 2^i,\ 
C_{U,L}(x)\le j\},
$$
where $x$ stands for $U(p)$.
The larger this set is the better parameters may have lists $L$ with small  $\CT(L|p)$.
If $\delta\ge \log|p|+O(1)$ then 
the list  $\bo^{i}$ for $i=C_U(U(p))$ and the list $\{p\}$ witness
that the set $S^{\delta}_p$ includes the entire gray set 
$P$ on the picture from Example~\ref{ex1}. 

The set $S^{\delta}_p$ may be much larger then the gray set $P$. For instance,
this happens when $p$ is a shortest program for $x=U(p)$. In this case 
the set $S^{\delta}_p$ coincides with the set of all points above 
the dashed line.
Are there infinitely many $p$
such that the set $S^{\delta}_p$ is close 
to the gray set $P$ in the picture?  In other words, are there infinitely many $p$
such that 
for every list $L$ with  $\CT_\Phi(L|p)\le\delta$ either $\log_2\# L>C_U(x)$,
or $C_{U,L}(x)\ge |p|$ (with certain accuracy)?
A  positive answer is provided by the following 

\begin{theorem}\label{th2}
Let $U$ be a standard machine.
For all $n$ and all $k>n$ there is a string $x$ with 
$C_U(x)=n+O(1)$ and its $U$-program $p$ of length at most $k+O(1)$ 
such that for all $\delta<k-\log k-O(1)$ and all $L$ with $\CT_\Phi(L|p)=\delta$ either $\# L\ge2^{n-\delta-\log k-O(1)}$ or
 $C_{U,L}(x)\ge k-1$.  
\end{theorem}

Notice that the inequality  $C_{U,L}(x)\ge k-1$ for 
$L=\{p\}$ implies that $|p|\ge k-O(1)$ and hence $|p|= k-O(1)$.

\section{The proofs}

We first drop in Theorems~\ref{th1} and~\ref{th1b}
the requirement $|p|\ge k$. As a reward, the lower bound for the 
list size will be a little bit stronger: $2^{|x|}-1$ in place of $2^{|x|}-2$.
                   
\begin{proof}[Proof of Theorem~\ref{th1}]
Let us first show that the statement of Theorem~\ref{th1} is invariant: if it holds 
for some standard machine $U$ then it holds for any other standard machine $U'$.
Indeed, assume that Theorem~\ref{th1} holds for a standard machine $U$.
To show Theorem~\ref{th1} for another standard machine  $U'$ and a list $L'(p')$,
choose a linearly bounded  translator $t$ from $U'$ to $U$ and 
a linearly bounded translator $s$ from $U$ to $U'$. Let $c'$ be a constant with
$|t(p')|\le |p'|+c'$.

Apply Theorem~\ref{th1} to  
the machine $U$ and the list $L(p)=t(L'(s(p)))$.
By Theorem~\ref{th1} for 
all $k$ there is a string  
$x$ and its $U$-program $p$ of length at most $k+c'+c$ such that  
$\# L(p)>2^{|x|}-1$ and 
the list $L(p)$ does not contain any $U$-program for $x$ of length less than $k+c'$. 

Let $p'=s(p)$. By construction, 
$$
|p'|\le  |p|+O(1)\le k+c'+c+O(1).
$$
We also have
$$
\# L'(p')\ge \# t(L'(p'))\ge 2^{|x|}-1.
$$
Finally the list $t(L(p'))$ does not contain any $U$-program of length less than 
$k+c'$ for $x$.
Hence the list $L(p')$ does not  contain any $U'$-program of length less than 
$k$ for $x$.

Thus it suffices to prove Theorem~\ref{th1} for the standard machine  
$U$ from Example~\ref{ex1},
that is for $U(\hat pq)=\Phi(p,q)$ where 
$\Phi$ is a G\"odel numbering   
of the family of computable functions of one variable.

We will let $p=\hat r q$ where $q$ is a string of length $k$ 
and $r$ does not depend on $k$.
The statement of the theorem will follow from the following
properties of $r,q$ 
and the function  $V\eqdef \Phi_r$ (of one variable):
\begin{itemize}
\item  $q$ is a $V$-program of a string $x$ such that 
\item $\# L(\hat rq)\ge 2^{|x|}-1$ and 
\item the list $L(\hat r q)$ contains no $U$-program for  $x$ of length less than $k$.  
\end{itemize}
Notice that the string $p=\hat rq$ has all the required properties.

It remains to find such $V,r$ and $q$.
The computable function $V$ and its $\Phi$-program $r$ will be defined using 
the Kleene fixed point theorem~\cite{rogers}.
By that theorem we may assume that computing $V$ we have access to a $\Phi$-program $r$
for $V$. 
We construct an algorithm that enumerates the graph of $V$. 

\textbf{The algorithm enumerating the graph of $V$}.
We maintain for all $k$  a string
$q_k$ of length $k$ and a string $x_k$. At the start let $q_k$ be any string of length $k$ and
let $x_k$ be the empty string.  
Enumerate  all  the pairs $\pair{q_k,x_k}$ into the graph of $V$ 
thus letting $V(q_k)=x_k$. 


Then we start an enumeration of the graph of  
$U$. Each time a new pair appears in that enumeration,
we look if the current situation is good or not.
We consider the current situation
\emph{good for } $k$ if 
the pair $\pair{q_k,x_k}$ has been enumerated  into the graph of $V$, 
$\# L(\hat r q_k)\ge2^{|x_k|}-1$ and 
the list $L(\hat r q_k)$ has no $\bar U$-program for $x_k$ of length less than $k$,
where $\bar U$ denotes the sub-function of $U$ consisting of all pairs enumerated so  far.  

At the start $\bar U=\emptyset$ and thus the situation is good for all $k$.
Each time a new pair appears in the enumeration of the graph of $U$,
we look whether the situation has become bad for some $k$.
Obviously that may happen only if a pair $\pair{s,x_k}$ with
$|s|<k$ and $s\in L(\hat r q_k)$ is enumerated. 
In that case pick a new string $q$ of length  $k$
(``new'' means that $q$ has not been used as $q_k$ earlier). 
Let $n$ be the integer with $2^{n+1}-1>\# L(\hat r q)\ge 2^{n}-1$.
For all strings $x$ of length at most $n$ consider the set $S(x)=\{s \mid \bar U(s)=x,\ |s|<k\}$
of $\bar U$-programs for $x$ of length less than $k$.
Pick any string $x$ of length at most $n$ such that $S(x)$ does not intersect the list  $L(\hat rq)$. 
As $\# L(\hat r q)< 2^{n+1}-1$ and the number of $x$'s is $2^{n+1}-1$, there is such $x$.
Then let $q_k=q$, $x_k=x$ and enumerate the pair $\pair{q,x}$
into the graph of $V$. The situation has become good for $k$.
\textbf{End of Algorithm.}

By Kleene's theorem for some $r$ this algorithm enumerates the
graph of the function $\Phi_r$. 
Let us show that for each  $k$, starting from some moment the situation is good for $k$.
Indeed, for any $k$ the situation may become bad less than $2^{k}$ times for $k$, as that may happen only after
a new pair of the form $\pair{s,x_k}$ with $|s|<k$ has appeared. 
On the other hand, the number of strings $q$ of length $k$ is $2^{k}$ 
and hence we indeed  are able to repair the situation $2^{k}-1$ times. 
\end{proof}

\begin{proof}[Proof of Theorem~\ref{th1b}]
Let  $\sing_\bot$ stand for the nowhere defined function.
There is a linearly bounded total computable translator $t$ mapping any $U$-program for $x\in\bo^*\cup\{\bot\}$ 
(for a standard machine $U$) to a $\phi$-program  for the function $\sing_x$.
There is also a  linearly bounded total computable translator $s$ mapping any $\phi$-program for $\sing_x$ 
back to a $U$-program for $x$. Given a list $L$ we just apply Theorem~\ref{th1} to the list $L'(p')=s(L(t(p')))$ and $k+c'$, 
where $c'$ is a constant with $|s(p)|\le |p|+c'$.  
\end{proof}

It remains to prove 
Theorems~\ref{th1} and~\ref{th1b} as they are stated, that is, 
with the requirement $|p|\ge k$ and with the lower bound 
$2^{|x|}-2$ for the list size.
Given any computable list $L(p)$ we add $p$ into the list and apply Theorem~\ref{th1} in the proven form
to the resulting list $L'(p)$. The list $L'(p)$ does have a $U$-program for $x$ of length
$|p|$ and has no $U$-program for $x$ of length less than $k$. This implies that $|p|\ge k$.
The program $p$ fulfills Theorem~\ref{th1} in the original form. 

The proof of  Theorem~\ref{th1b} is entirely similar.

\begin{remark}
As Jason Teutsch observed, one can prove Theorem~\ref{th1} without using the
fixed point theorem. To this end we modify the construction of $V$ so that 
$V$ becomes a standard machine. Specifically, we first let $V_{0q}=U_q$  
for all strings $0q$ starting with zero, where $U$ is any standard machine. 
Then we define $V_{1q}$ so that   
for all $k$ there is a string $1q$ of length $k+1$ such that
\begin{itemize}
\item  $1q$ is a $V$-program of a string $x$ such that 
\item $\# L(1q)\ge 2^{|x|}-1$ and 
\item the list $L(1q)$ contains no $V$-program for  $x$ of length less than $k$.  
\end{itemize}
This can be done by the same technique. The function $V$ defined in this way 
satisfies the theorem. As we already observed, this implies that the theorem holds for
all standard machines.
\end{remark}

\begin{proof}[Proof of Theorem~\ref{th2}]
The proof is very similar to that of Theorem~\ref{th1}.
We construct a computable function  $V$ such that
for all $k>n$ there are strings  $q,x$ of lengths $k,n$, respectively, 
with
\begin{itemize}
\item  $V(q)=x$,
\item $C_U(x)>n-1$,
\item for all $\delta<k-\log k-2$ and all $L$ with $\CT_\Phi(L|q)=\delta$ and $\#L<2^{n-\delta-\log k-1}$ we have 
$C_{U,L}(x)\ge k-1$.
\end{itemize}

\textbf{The algorithm enumerating the graph of $V$}.
This time we maintain for all $k$  a bunch of pairs 
$\{(q_{kn},x_{kn})\mid n=0,1,\dots,k-1\}$. The length of $q_{kn}$ is
$k$ and the length of  $x_{kn}$ is $n$. At the start let $q_{kn}$ be the $n$th string of length $k$ and
let $x_{kn}$ be the first string of length $n$ (independent of $k$).  
Enumerate  all  the pairs $\pair{q_{kn},x_{kn}}$ into the graph of $V$ 
thus letting $V(q_{kn})=x_{kn}$. 

Then we start an enumeration of the graph of  
$U$ and an enumeration of the graph of $\Phi$. We denote  by $\bar U$ and $\bar \Phi$
the sub-functions of $U$ and $\Phi$ consisting of all pairs (triples) enumerated so  far. 
For each $k$ we look if the situation
is \emph{good for} $k$. This means that for all $n<k$  
the pair $\pair{q_{kn},x_{kn}}$ has been enumerated  into the graph of $V$, 
$C_{\bar U}(x_{kn})\ge n-1$ and for all $\delta<k-\log k-2$ and all $L$ with $\CT_{\bar\Phi}(L|q_{kn})=\delta$ and 
$\#L<2^{n-\delta-\log k-1}$ 
we have 
$C_{\bar U,L}(x_{kn})\ge k-1$. Here $\CT_{\bar\Phi}(L|q)$ means the minimal length of $p$ 
such that $\bar\Phi_p$ is defined \emph{on all $q'$ of length $k$} 
and $\bar\Phi_p(q)=L$.

At the start $\bar U$ and $\bar\Phi$ are empty and thus the situation is good
for all $k$.
Each time a new pair (triple) appears in the enumeration of the graphs of $U$ or $\Phi$,
we look whether the situation has become bad for some $k$.
This may may happen only if $C_{\bar U}(x_{kn})$ has become less than $n-1$ (for some $n<k$) or a new list
$L$ with $\CT_{\bar\Phi}(L|q_{kn})<k-\log k-2$ 
appeared or for an old list $L$ the value $C_{\bar U,L}(x_{kn})$ has become less than $k-1$ (for some $n<k$).
In all the cases we first change
$q_{kn}$ and then we change $x_{kn}$. The string 
$q_{kn}$ is replaced by 
any  a new string $q$ of length  $k$
(``new'' means that $q$ has not been used as $q_{k*}$ earlier). 
The string $x_{kn}$ is replaced by any string $x$ of length  $n$ 
such that $C_{\bar U}(x)\ge n-1$ and 
the set $S(x)=\{s \mid \bar U(s)=x,\ |s|<k-1\}$
does not intersect the union (over all $\delta$) of all lists  $L$ of cardinality less than $2^{n-\delta-\log k-1}$ with  
$\CT_{\bar\Phi}(L|q)=\delta$.
Notice that for every $p$ there is only one list $L$ with $\bar\Phi_p(q)=L$
thus the total number of strings in all these lists is less than $\sum_{\delta<k}2^{\delta}\cdot 2^{n-\delta-\log k-1}=2^{n-1}$.
On the other hand, the number of strings $x$ of length  $n$ 
with  $C_{\bar U}(x)\ge n-1$ is more than $2^{n-1}$. Thus there is such $x$. 

Then let $q_{kn}=q$, $x_{kn}=x$ and enumerate the pair $\pair{q,x}$
into the graph of $V$. The situation has become good for $k$.
\textbf{End of Algorithm.}

We have to show that we are able to choose a new string of length $k$ each time we need one.
Any replacement of a string of the form $q_{k*}$ is caused by 
\begin{itemize}
\item discovering a new $p$ of length less than $k-2\log k-2$ such that $\Phi_p$ is defined on all strings of length $k$
(this may cause replacement of the whole bunch of $q_{kn}$'s, for all $n<k$), or 
\item  discovering  a new $U$-program $r$ of length less than $k-1$,
which may cause the replacement of $q_{kn}$ only if $U(r)=x_{kn}$ thus for a single  $n$, or
\item discovering a new halting  $U$-program of length less than $n-1$ for $x_{kn}$, 
which again may cause the replacement of $q_{kn}$ only for a single  $n$.
\end{itemize}

Thus the total number of strings $q_{k*}$ we need is less than  
$$
k+\sum_{\delta<k-\log k-2}k2^{\delta}+2^{k-1}<2^k.
$$

To prove the theorem 
let $p$ be the $U$-program of $x$ obtained from $q$ by translation from $V$ to $U$.
Then $|p|\le k+O(1)$. Notice that $\CT_\Phi(L|q)\le \CT_\Phi(L|p)+O(1)$. Indeed, let  
$s$ be a linearly bounded translator from $V$ to $U$. Then $\Phi(r,s(q))$ is a computable function
hence there is a total computable function $t$ with $\Phi_{t(r)}(q)=\Phi_r(s(q))$.
If $\Phi_r$ is total then so is $\Phi_{t(r)}$. Hence $\CT_\Phi(L|q)\le \CT_\Phi(L|s(q))+O(1)$. 
\end{proof}

\textbf{Acknowledgments.}
The author is sincerely grateful to Alexander Shen for asking the question and
hearing the preliminary version of the proof of the result.
The author is grateful to Jason Teutsch for the idea of how to omit  the use of the fixed point theorem.
The author is grateful to Alexander Rubtsov for discovering ``intermediate'' lists. 
The author is also grateful to the hospitality of the IMS of University of Singapore.


\begin{thebibliography}9

\bibitem{bmvz}
 Bruno Bauwens, Anton Makhlin, Nikolay Vereshchagin, Marius Zimand. Short lists with short programs in short time.
In Proceedings 28-th
IEEE Conference on Computational Complexity (CCC), Stanford, CA, pages 98--108, June 2013.
ECCC report TR13-007. 

\bibitem{tz}
Jason Teutsch and Marius Zimand. On approximate decidability of minimal programs. 2014. 
Available from http://arxiv.org/abs/1409.0496 and http://people.cs.uchicago.edu/~teutsch/papers/teutschpubs.html. 

\bibitem{rogers}
Hartley Rogers, Jr., The Theory of Recursive Functions and Effective Computability, MIT Press, 1987.

\bibitem{schnorr}
C.P. Schnorr.
\newblock Optimal enumerations and optimal {G}{\"o}del numberings.
\newblock {\em Mathematical Systems Theory}, 8(2):182--191, 1975.

\bibitem{s}
Alexander Shen.
A talk on a meeting during the IMS program ``Algorithmic Randomness'' (IMS, University of Singapore, 2--30 June 2014). 


\end{thebibliography}
\end{document}